\def\timestamp{%
Time-stamp: <hicomp-amsart-2.tex: Thursday 21-08-2008 at 10:08:22 (cest)>}
\def\stripname Time-stamp: <#1 #2>{#2}
\edef\filedate{\expandafter\stripname\timestamp}
\documentclass{amsart}

\usepackage{amssymb,amsrefs,amsthm}

\DeclareMathSymbol\restr\mathbin{AMSa}{"16}
\DeclareMathSymbol\le   \mathrel{AMSa}{"36}  

\theoremstyle{plain}
\newtheorem{theorem}{Theorem}[section]
\newtheorem{lemma}[theorem]{Lemma}
\newtheorem{corollary}[theorem]{Corollary}
\newtheorem{proposition}[theorem]{Proposition}

\theoremstyle{definition}
\newtheorem*{KM}{Property (KM)}

\theoremstyle{remark}
\newtheorem{remark}[theorem]{Remark}

\numberwithin{equation}{section}

\newcommand\cl{\operatorname{cl}}

\begin{document}
\title[On hereditarily indecomposable compacta]
       {On hereditarily indecomposable compacta and factorization of maps}

\author[K. P. Hart and E. Pol]{Klaas Pieter Hart and El\.{z}bieta Pol}

\address[Klaas Pieter Hart]
        {Faculty of Electrical Engineering, Mathematics and 
         Computer Science\\
         TU~Delft\\
         Postbus 5031\\
         2600~GA~~Delft\\
         the Netherlands}
\email{k.p.hart@tudelft.nl}
\urladdr{http://fa.its.tudelft.nl/\~{}hart}

\address[El\.{z}bieta Pol]
        {Institute of Mathematics\\
         University of Warsaw\\
         Banacha 2\\
         02-197 Warszawa\\
         Poland}
\email{pol@mimuw.edu.pl}
\thanks{The second author was partially supported by 
        MNiSW Grant Nr N201 034 31/2717}

\subjclass[2000]{Primary 54F15. Secondary:  54F45 03C98}

\keywords{hereditarily indecomposable compacta, 
          \v{C}ech-Stone compactification, factorization, 
          L\"owenheim-Skolem theorem}

\begin{abstract}
We prove a general factorization theorem for maps with hereditarily
indecomposable fibers and apply it to reprove a theorem of 
Ma\'ckoviak on the existence of universal hereditarily indecomposable
continua.
\end{abstract}

\date{\filedate}

\maketitle

\section{Introduction}

All spaces are assumed to be normal. 
By a map we mean a continuous function. 
We say that a compactum~$X$ is \emph{hereditarily indecomposable} 
if for every two intersecting continua in~$X$ one is contained in the other. 
The main result of this note is the following theorem.

\begin{theorem}\label{thm.1.1} %% 1.1.Theorem. 
Let $f : X \to Y$ be a perfect map with hereditarily indecomposable fibers 
from a separable metrizable space~$X$ onto a zero-dimensional separable 
metrizable space~$Y$. 
Then there are a hereditarily indecomposable metrizable 
compactification~$X^{\star }$ of\/~$X$ with $\dim X^*= \dim X$ and 
a zero-dimensional metrizable compactification~$Y^*$ of\/~$Y$ such 
that $f$~can be extended to a map $f^*: X^* \to Y^*$. 
\end{theorem}

Let us note that this result, combined with a pseudosuspension
method, yields a theorem of Ma\'{c}kowiak~\cite{M2} on the
existence of universal $n$-dimensional hereditarily
indecomposable continua. 
This theorem was obtained by Ma\'{c}kowiak by a quite different method 
based on a subtle use of inverse limits. 
We comment on this in Corollary~\ref{cor.4.1}.

Rather unexpectedly, our proof uses, in an essential way, large nonmetrizable 
compactifications and a considerable strenghtening of Marde\v{s}i\'c's
Factorization Theorem (see \cite{E}*{Theorem~3.4.1}).
This strengthening is a dual version of the L\"owenheim-Skolem theorem from
model theory; it appears as Theorem~3.1 in~\cite{MR1785837} and it
was put to good use in~\cite{HvMP} and~\cite{vdS}.
In Section~\ref{sec.2} we explain some general facts concerning this 
technique and in section~\ref{sec.3} we show how our theorem follows from 
these results. 
Among other consequences of this technique is the following theorem, proved in
section~\ref{sec.3}.

\begin{theorem}\label{thm.1.2} %% 1.2.Theorem.
For every cardinal~$\tau$ and $n \in \{0,1,\ldots , \infty \}$ 
there exists a hereditarily indecomposable compactum $X(n,\tau)$ of 
weight~$\tau$ and dimension~$n$ that contains a copy of every hereditarily
indecomposable compactum of weight not more than~$\tau$ and 
dimension at most~$n$.
\end{theorem}

The following property of a space $X$ was formulated by
Krasinkiewicz and Minc~\cite{KM}:

\begin{KM}
For every two disjoint closed
sets $C$ and $D$ in $X$ and disjoint open sets $U$ and $V$ in $X$
with $C \subset U$ and $D \subset V$ there exist closed sets
$X_0$, $X_1$ and $X_2$ in $X$ such that 
$X = X_0 \cup X_1 \cup X_2 $, 
$C \subset X_0$, 
$D \subset X_2$, 
$X_0\cap X_1 \subset V$, 
$X_1 \cap X_2 \subset U$ and 
$X_0\cap X_2= \emptyset$.
\end{KM}

To avoid having to write down the six conditions each time we shall
call a triple $\langle X_0,X_1,X_2\rangle$ a \emph{fold of~$X$} for
the quadruple $\langle C,D,U,V\rangle$.

%The space $X$ satisfying condition (KM) does not need to be
%compact.

\begin{theorem}[\cite{KM}] \label{thm.1.3} %%1.3.Theorem (see \cite{KM}) 
A compact space is hereditarily indecomposable 
if and only if it has Property~(KM).
\end{theorem}

\section{A factorization method}
\label{sec.2}

The factorization method alluded to in the Introduction is based on a mix
of Model Theory and Set-Theoretic Topology.
It works best in the realm of compact Hausdorff spaces, as will become clear
shortly.

The first ingredient is Wallman's representation theorem, \cite{Wallman},
for distributive lattices: if $L$ is such a lattice then the set $wL$ of
ultrafilters on~$L$ carries a natural compact $T_1$-topology.
This topology has the family $\{\bar a:a\in L\}$ as a base for the closed
sets, where $\bar a=\{u\in wL:a\in u\}$.

If $X$ is compact and $T_1$ and $L$ is the family of closed subsets of~$X$,
with union and intersection as its lattice operations then 
$x\mapsto u_x=\{a\in L:x\in a\}$ is a homeomorphism from~$X$ onto~$wL$;
this remains true if $L$~is a base for the closed sets of~$X$ that is closed
under unions and intersections.
See, e.g., \cite{Aarts} for a short introduction to Wallman representations.

For a normal space~$X$ one can obtain the \v{C}ech-Stone compactification,
$\beta X$, as the Wallman representation of the lattice of closed sets 
of~$X$.
This is the key to the next theorem.

\begin{theorem}\label{thm.2.1} %% 2.1.Theorem.
If $X$ has Property~(KM) then so does its \v{C}ech-Stone 
compactification~$\beta X$ and, in particular, 
$\beta X$~is hereditarily indecomposable.
\end{theorem}

\begin{proof}
To begin: it should be clear that Property~(KM) can be 
(re)for\-mu\-lated in terms 
of closed sets only and that it is a finitary lattice-theoretic property; 
one can express it as an implication involving seven variables.
Thus if $X$~has Property~(KM) then the canonical base, $\mathcal{B}$, for 
the closed sets of~$\beta X$ satisfies this implication.
This does not automatically imply that $\beta X$ has Property~(KM), because
that means that the full family of closed sets of~$\beta X$ satisfies the
lattice-theoretic formula.
However, in the present case one can start with arbitrary 
$C$, $D$, $U$ and~$V$ and use compactness and the fact that $\mathcal{B}$
is closed under finite unions and intersections to 
find~$C'$, $D'$, $U'$ and~$V'$ such 
that $C\subseteq C'\subseteq U'\subseteq U$ and  
$D\subseteq D'\subseteq V'\subseteq V$, and such that 
$C'$, $D'$, $\beta X\setminus U'$ and $\beta X\setminus V'$ belong 
to~$\mathcal{B}$.
One can then find a fold $\langle X_0,X_1,X_2\rangle$ for 
$\langle C',D',U',V'\rangle$ in~$\mathcal{B}$  
and this will also be a fold for~$\langle C,D,U,V\rangle$.
\end{proof}

The second ingredient is the use of notions from Model Theory,
especially elementary substructures and the L\"owenheim-Skolem theorem.
In the context of lattices elementarity is perhaps best explained in terms
of solutions to equations.
One can interpret Property~(KM) as saying that certain equations should have
solutions: the quadruple $\langle C,D,U, V\rangle$ determines six equations
and a fold  $\langle X_0,X_1,X_2\rangle$ is a solution to this system.

One calls $M$ an elementary sublattice of~$L$ if every lattice-theoretic
equation with constants from~$M$ that has a solution in~$L$ also has a 
solution in~$M$.

To illustrate its use we prove the following lemma.

\begin{lemma}\label{lemma.elem.KM}
Assume $X$ is a hereditarily indecomposable compact space and let $L$ be 
an elementary sublattice of the lattice of closed subsets of~$X$.
Then $wL$ is also hereditarily indecomposable.
\end{lemma}

\begin{proof}
By elementarity the lattice~$L$ satisfies Property~(KM): 
if $C$, $D$, $X\setminus U$ and $X\setminus V$ belong to~$L$ then there is 
a fold $\langle X_0,X_1,X_2\rangle$ in the full family of closed sets, 
\emph{hence} there is also such a fold in~$L$.

Next, in $wL$ the same argument as in the proof of Theorem~\ref{thm.2.1} 
applies: an arbitrary quadruple can be expanded to a quadruple from the base.
\end{proof}

The L\"owenheim-Skolem Theorem provides us with many elementary substructures:
given a lattice~$L$ and some subset~$A$ of~$L$ one can construct 
an elementary sublattice~$L_A$ of~$L$ that contains~$A$ and whose
cardinality is at most $|A|\times\aleph_0$.

\begin{theorem}[\cites{MR1785837,vdS,HvMP}]\label{thm.2.2} 
%% 2.2.Theorem. (\cite{vdS},\cite{HvMP})} 
Let $f : X \to Y$ be  a continuous surjection from a hereditarily 
indecomposable compact  space onto a compact space. 
Then there are a compact space~$Z$ and continuous maps
$g:X \to Z$ and $h : Z \to Y$ such that 
$Z$~is hereditarily indecomposable, 
$\dim Z= \dim X$,
\ $w(Z)=w(Y)$ and $f=h \circ g$.
\end{theorem}

\begin{proof}
Let $\mathcal{B}$ be a base for the closed sets of~$Y$ of cardinality~$w(Y)$.
Via $B\mapsto f^{-1}[B]$ we can identify $\mathcal{B}$ with a sublattice
of the lattice~$\mathcal{D}$ of closed subsets of~$X$.

Apply the L\"owenheim-Skolem Theorem to find an elementary 
sublattice~$\mathcal{C}$ of~$\mathcal{D}$ that contains~$\mathcal{B}$
and has the same (infinite) cardinality as~$\mathcal{B}$; 
we let $Z=w\mathcal{C}$.
The two inclusions 
$\mathcal{B}\subseteq\mathcal{C}\subseteq\mathcal{D}$ induce
continuous surjections $g:X\to Z$ and $h:Z\to Y$ that, as one readily shows,
satisfy $f=h\circ g$.
By Lemma~\ref{lemma.elem.KM} the space~$Z$ is hereditarily indecomposable.
The same argument shows that $\dim Z=\dim X$: one can use, for example,
the Theorem on Partitions, \cite{E}*{Theorem~1.7.9}, to turn the statement
$\dim X\le n$ into an equation~$\Phi_n$.
By elementarity $\mathcal{C}$ and $\mathcal{D}$ satisfy~$\Phi_n$ for exactly
the same values of~$n$.
The expansion trick applies in this case as well so that $\dim Z\le n$ for
exactly the same values of~$n$ for which $\mathcal{C}$ satisfies~$\Phi_n$.
\end{proof}

We refer to \cite{HodgesShorterModelTheory} for basic information on Model
Theory.

\begin{remark}
The thesis \cite{vdS} contains a systematic study of properties that are
preserved by continuous maps that are induced by elementary embeddings.  
\end{remark}

\section{Proofs of the main results}
\label{sec.3}

We start with the following

\begin{theorem}\label{thm.3.1}  %% 3.1.Theorem.
Let  $f: E \to F$ be a perfect mapping from a space $E$ onto a 
strongly zero-dimensional paracompact space~$F$ such that for 
every $y \in F$ the fiber~$f^{-1}(y)$ is hereditarily indecomposable.
Then $E$~has  Property~(KM). 
\end{theorem}

\begin{proof} 
Let $C$ and~$D$ be disjoint closed subsets of $E$ and let $U$ and~$V$
disjoint open subsets of $E$ around $C$ and $D$ respectively.

Let us fix $y \in F$. 
We shall find a (clopen) neighbourhood~$O_y$ of~$y$ and a fold 
of~$f^{-1}[O_y]$ for $\langle C\cap f^{-1}[O_y],D\cap f^{-1}[O_y],U,V\rangle$.
Since $f^{-1}(y)$ is compact and hereditarily indecomposable, 
by Theorem~\ref{thm.1.3} it has Property~(KM) and hence there exists a 
fold $\langle X_0,X_1,X_2\rangle$ of~$f^{-1}(y)$ for
$\langle C\cap f^{-1}(y),D\cap f^{-1}(y),U,V\rangle$.

Apply \cite{E}*{Theorem~3.1.1} to find a sequence
$B=\langle W_0,W_1, W_2, O_U, U_V\rangle$ of open sets 
such that their closures form a swelling of the 
sequence~$A=\langle C\cup X_0,X_1,D\cup X_2,X\setminus U,X\setminus V\rangle$,
which means that each term of~$A$ is a subset of the corresponding term~$B$ 
and whenever $I$ is such that $\bigcap_{i\in I}A_i=\emptyset$ then
$\bigcap_{i\in I}\cl B_i=\emptyset$.
Specifically this means that 
\begin{enumerate}
\item $f^{-1}(y)\subseteq W_0\cup W_1\cup W_2$;
\item $\cl W_0\cap\cl W_1\subseteq V$;
\item $\cl W_1\cap\cl W_2\subseteq U$;
\item $\cl W_0\cap\cl W_2=\emptyset$.
\end{enumerate}
As the map~$f$ is perfect and the space~$F$ is zero-dimensional we can find
a clopen neighbourhood~$O_y$ of~$y$ such that 
$f^{-1}[O_y]\subseteq W_0\cup W_1\cup W_2$.
It follows that $\langle\cl W_0,\cl W_1,\cl W_2\rangle$ is a fold 
of~$f^{-1}[O_y]$ for
$\langle C\cap f^{-1}[O_y], D\cap f^{-1}[O_y],U,V\rangle$.

By strong zero-dimensionality and paracompactness we can find a disjoint
clopen refinement~$\mathcal{O}$ of~$\{O_y:y\in F\}$; it is then a routine
matter to combine the `local' folds into one `global' fold of~$E$ 
for~$\langle C,D,U,V\rangle$.
\end{proof}

We are now ready to prove the first main result.

\begin{proof}[Proof of Theorem~\ref{thm.1.1}]
To begin we construct a zero-dimensional compactification~$Y^*$ of~$Y$,
a compactification~$X_1$ of~$X$ and a continuous extension $f_1:X_1\to Y^*$.

One way of doing this is by assuming that $X$~is embedded in the Hilbert 
cube~$I^{\aleph_0}$, that $Y$~is embedded in the Cantor set~$\{0,1\}^{\aleph_0}$
and then to identify $X$ with the graph of~$f$, i.e.,
$X$~is identified with
$G(f)=\bigl\{\bigl(x,f(x)\bigr):x\in X\bigr\}
     \subseteq I^{\aleph_0}\times\{0,1\}^{\aleph_0}$
via $x\mapsto\bigl(x,f(x)\bigr)$.
After this identification $f$~is simply $\pi_2\restr G(f)$, where $\pi_2$~is
the projection onto the second factor of the product;
we can then let $X_1=\cl G(f)$ (in the product) and $Y^*=\cl Y$ 
(in the Cantor set), the desired extension~$f_1$ of~$f$ then is 
$\pi_2\restr X_1$.

Next let $j:\beta X\to X_1$ be the natural map 
(the extension of the inclusion of~$X$ into~$X_1$).
By Theorem~\ref{thm.3.1} $X$~has Property~(KM) so by Theorem~\ref{thm.2.1} 
$\beta X$~is hereditarily indecomposable.
Apply Theorem~\ref{thm.2.2} to obtain a factorization of~$j$ consisting
of maps $g:\beta X\to X^*$ and $h:X^*\to X_1$ in which $X^*$~is hereditarily 
indecomposable, second-countable and satisfies $\dim X^*=\dim \beta X=\dim X$. 
Then $X^*$ is a metrizable compactification of~$X$ as $g\restr X$~is a 
homeomorphism.
It remains to set $f^*=f_1\circ h$.
\end{proof}

Let us note that since $f$ is perfect and $X^*$ is a compactification of~$X$,
the extension $f^*$ satisfies $(f^*)^{-1}(y)=f^{-1}(y)$ for $y \in Y$.

To get universal hereditarily indecomposable compacta we use the
the factorization method again.

\begin{proof}[Proof of Theorem~\ref{thm.1.2}] 
Let $\{X_{s}\}_{s\in S}$ be the family of all compact hereditarily 
indecomposable subspaces of the Tychonoff cube~$I^{\tau}$ whose dimension is 
not larger than~$n$, and let $i_{s}: X_{s}\to I^{\tau}$ be the inclusion. 
Let $X = \bigoplus _{s \in S}X_{s}$ be the free union of the~$X_{s}$'s and 
let $i : X \to I^{\tau}$ be defined by $i(x)=i_{s}(x)$ for $x \in X_{s}$. 
Let $f : \beta X \to I^{\tau}$ be the extension of~$i$ over~$\beta X$. 
Obviously, $X$~has Property~(KM), so by Theorem~\ref{thm.2.1} 
$\beta X $~is hereditarily indecomposable. 
By Theorem~\ref{thm.2.2} $f$~can be factored as $h \circ g$,  
where $g:X \to Z$ and $h : Z \to Y$ and where $Z$~is hereditarily 
indecomposable,  $w(Z)\le\tau$ and $\dim Z = \dim X$. 
We can take $X(n,\tau)=Z$.
\end{proof}

\section{Corollaries and Remarks}
\label{sec.4}

%% 4.1.Remark.
Let us note that as a corollary to either Theorem~\ref{thm.1.1} or
Theorem~\ref{thm.1.2}
one can obtain the following theorem of Ma\'{c}kowiak \cite{M2}.

\begin{corollary}\label{cor.4.1} 
For every $n \in \{1,2,\ldots \infty \}$ there exists a
hereditarily indecomposable metric continuum $Z_{n}$ of dimension
$ n$ containing a copy of every hereditarily indecomposable metric
continuum of dimension at most~$n$.
\end{corollary}

\begin{proof}[Proof using Theorem~\ref{thm.1.1}]
Let $\mathcal{P}$ be the subset of the hyperspace $2^{I^{\aleph _0}}$ 
of the Hilbert cube consisting of all hereditarily indecomposable continua 
of dimension~$n$ or less. 
Then $\mathcal{P}$ is a  $G_\delta$-subset of~$2^{I^{\aleph_0}}$
(see  \cite{Ku}*{\S~45, IV, Theorem~4 and \S~48, V, Remark~5}).
Therefore there is a continuous surjection $\varphi:Y\to \mathcal{P}$, 
where $Y$~is the space of the irrationals. 
Then let $X$ be the following subset of $I^{\aleph_0} \times Y$:
$$
\bigl\{ (x,t): t \in Y \text{  and }x\in\varphi(t)\bigr\}
$$
and let $\pi:  I^{\aleph_0}\times Y\to Y$ be the projection.
The restriction $f=\pi \restr X : X \to Y$ is a perfect map 
(cf.~\cite{Ku1}*{\S~18} or~\cite{vM}*{Exercise~1.11.26}) 
with hereditarily indecomposable fibers. 
By Theorem~\ref{thm.1.1} there exists a hereditarily indecomposable 
$n$-dimensional compact space~$X^*$ that contains~$X$ and hence a copy
of every hereditarily indecomposable continuum of dimension~$n$.

The decomposition of~$X^*$ into its components yields a compact 
zero-di\-men\-sional space.
The pseudo-arc $P$ contains a copy of this decomposition space (as indeed does
any uncountable compact metrizable space).
Let $q:X^*\to P$ be a map such that $A=q[X]$ is that decomposition space
and $q:X\to A$ is the quotient map.
 
By Theorem~15 of~\cite{M1} there exist a hereditarily indecomposable 
continuum~$Z_{n}$  and an atomic mapping~$r$ from~$Z_{n}$ onto~$P$ such 
that $r \restr r^{-1}(P\setminus A )$ is a homeomorphism and $r^{-1}(A)$~is 
homeomorphic to~$X^*$ 
\ ($Z_{n}$~is a so-called pseudosuspension of~$X^*$ over~$P$ by~$q$). 
Since $\dim Z_{n}\le n$ by the countable sum theorem and $Z_{n}$~contains 
$X^*$ topologically, the space $Z_{n}$ has the required properties.
\end{proof}

\begin{proof}[Proof using Theorem~\ref{thm.1.2}]
Use the second half of the previous proof but now take the pseudosuspension 
of the space $X(n, \aleph_0)$ over~$P$ by~$q$, 
where $q : X(n, \aleph_0) \to P$ is a quotient map such that 
$A=q[X(n, \aleph_0)]$ is the decomposition space of~$X(n, \aleph_0)$ 
into its components.
\end{proof}

\begin{remark}%% 4.2.Remark.
If one uses Theorem~\ref{thm.2.2} instead of Marde\v{s}i\'{c}'s
Factorization Theorem, 
and standard topological reasoning 
(see \cite{E}*{proofs of Theorems~5.4.3 and~3.4.2}) 
one gets the following results.

\begin{proposition} % \smallskip\noindent \textbf{(A)} {\itshape 
For every  hereditarily indecomposable compact space~$X$ such 
that $\dim X = n$ and the weight of~$X$ is equal to~$\tau$, there exists 
an inverse system $\mathbf{S}=\{X_{\sigma}, \pi_{\rho}^{\sigma}, \Sigma \}$, 
where $|\Sigma|\le \tau$, of metrizable hereditarily indecomposable 
compact spaces of dimension~$n$ whose limit is homeomorphic to~$X$. 
If $X$~is a continuum, then all~$X_\sigma$ are continua.
\end{proposition}

\begin{proposition}  %\smallskip\noindent \textbf{(B)} {\itshape 
Every normal $n$-dimensional space~$X$ of weight~$\tau$ 
that has Property~(KM) has a hereditarily indecomposable
compactification $\tilde{X}$ of dimension~$n$ and of weight~$\tau$.
% with $\dim\tilde{X}\le n$ and the weight of $\tilde{X}
%\le \tau$.
\end{proposition}
\end{remark}

\begin{remark}
The results of this paper remain valid if in the formulation of Property~(KM)
one replaces closed sets by zero-sets and open sets by cozero-sets.
This implies that in Theorem~\ref{thm.2.1} one can relax the assumption
of normality to complete regularity.
\end{remark}

\end{document}